\documentclass{amsart}
\usepackage{amssymb}
\usepackage{amsfonts}
\usepackage{amsmath}
\usepackage{graphicx}
\usepackage[pagebackref]{hyperref}
\theoremstyle{plain}
\newtheorem{thm}{Theorem}[section]
\newtheorem{corollary}[thm]{Corollary}
\newtheorem{lemma}[thm]{Lemma}
\newtheorem{proposition}[thm]{Proposition}

\theoremstyle{definition}
\newtheorem{definition}[thm]{Definition}

\theoremstyle{remark}

\newtheorem{example}[thm]{Example}

\numberwithin{equation}{section}

\def\1{{\rm (1)}}
\def\2{{\rm (2)}}
\def\3{{\rm (3)}}
\def\4{{\rm (4)}}
\def\5{{\rm (5)}}

\begin{document}

\title[Ratliff-Rush Closure of Ideals in Integral Domains]
{Ratliff-Rush Closure of Ideals in Integral Domains}


\author{A. Mimouni}

\address{Department of Mathematical Sciences, King Fahd University of Petroleum \& Minerals, P.O. Box 5046, Dhahran 31261, KSA}

\email{amimouni@kfupm.edu.sa}

\thanks{This work was funded by KFUPM under Project \# FT070001.}



\begin{abstract}
This paper studies the Ratliff-Rush closure of ideals in integral
domains. By definition, the Ratliff-Rush closure of an ideal $I$ of
a domain $R$ is the ideal given by
$\tilde{I}:=\bigcup(I^{n+1}:_{R}I^{n})$ and an ideal $I$ is said to
be a Ratliff-Rush ideal if $\tilde{I}=I$. We completely characterize
integrally closed domains in which every ideal is a Ratliff-Rush
ideal and we give a complete description of the Ratliff-Rush closure
of an ideal in a valuation domain.
\end{abstract}

\maketitle

\section{Introduction}\label{Int}

Let $R$ be a commutative ring with identity and $I$ a regular ideal
of $R$, that is, $I$ contains a nonzero divisor. The ideals of the
form $(I^{n+1}:_{R}I^{n}):=\{x\in R| xI^{n}\subseteq I^{n+1}\}$
increase with $n$. In the case where $R$ is a Noetherian ring, the
union of this family is an interesting ideal, first studied by
Ratliff and Rush in \cite{RR}. In \cite{HLS}, W. Heinzer, D. Lantz
and K. Shah called the ideal $\tilde{I}:=\bigcup(I^{n+1}:_{R}I^{n})$
the Ratliff-Rush closure of $I$, or the Ratliff-Rush ideal
associated with $I$. An ideal $I$ is said to be a Ratilff-Rush
ideal, or Ratliff-Rush closed, if $I=\tilde{I}$. Among the
interesting facts of this ideal is that, for any regular ideal $I$
in a Noetherian ring $R$, there exists a positive integer $n$ such
that for all $k\geq n$, $I^{k}=(\tilde{I})^{k}$, that is, all
sufficiently high powers of a regular ideal are Ratliff-Rush ideals,
and a regular ideal is always a reduction of its Ratliff-Rush
closure in the sense of Northcoot-Rees (see \cite{NR}), that is,
$I(\tilde{I})^{n}=(\tilde{I})^{n+1}$ for some positive integer $n$.
Also the ideal $\tilde{I}$ is always between $I$ and the integral
closure $I'$ of $I$, that is, $I\subseteq \tilde{I}\subseteq I'$,
where $I':=\{x\in R| x$ satisfies an equation of the form $x^{k}+
a_{1}x^{k-1}+\dots + a_{k}=0$, where $a_{i}\in I^{i}$ for each $i\in
\{1, \dots, k\}\}$. Therefore, integrally closed ideals, i. e.,
ideals such that $I=I'$, are Ratliff-Rush ideals. Since then, many
investigations of the Ratliff-Rush closure of ideals in a Noetherian
ring have been carried out, for instance, see \cite{HJLS},
\cite{HLS}, \cite{Li}, \cite{RS} etc. The purpose of this paper is
to extend the notion of Ratliff-Rush closure of ideals to an
arbitrary integral domain and examine ring-theoretic properties of
this kind of closure. In the second section, we give an answer to a
question raised by B. Olberding \cite{O4} about the classification
of integral domains for which every ideal is a Ratliff-Rush ideal in
the context of integrally closed domains. This lead us to give a new
characterizations of Pr\"ufer and strongly discrete Pr\"ufer
domains.  Specifically, we prove that ``a domain $R$ is a Pr\"ufer
(respectively strongly discrete Pr\"ufer) domain if and only if $R$
is integrally closed and each nonzero finitely generated
(respectively each nonzero) ideal of $R$ is a Ratliff-Rush ideal"
(Theorem~\ref{RRID.6}). It turns that a Ratliff-Rush domain (i. e.,
domain such that each nonzero ideal is a Ratliff-Rush ideal) is a
quasi-Pr\"ufer domain, that is, its integral closure is a Pr\"ufer
domain. As an immediate consequence, we recover Heinzer-Lantz-Shah's
results for Noetherian domains (Corollary~\ref{RRID.8}). The third
section deals with valuation domains. Here, we give a complete
description of the Ratliff-Rush closure of a nonzero ideal in a
valuation domain (Proposition~\ref{RRIV.2}), and we state necessary
and sufficient condition under which the Ratliff-Rush closure
preserves inclusion (Proposition~\ref{RRIV.3}).  We also extend the
Ratliff-Rush closure to arbitrary nonzero fractional ideals of a
domain $R$, and we investigate its link to the notions of star
operations. We prove that ``for a valuation domain $V$, the
Ratliff-Rush closure is a star operation if and only if every
nonzero nonmaximal prime ideal of $V$ is not idempotent, and in this
case it coincides with the $v$-closure" (Theorem~\ref{RRIV.5}).\\

Throughout, $R$ denotes an integral domain, $qf(R)$ its quotient
field, and  $R'$ and $\overline{R}$ its integral closure and
complete integral closure respectively. For a nonzero (fractional)
ideal $I$ of $R$, the inverse of $I$ is given by
$I^{-1}=(R:I):=\{x\in qf(R)| xI\subseteq R\}$. The $v$-closure and
$t$-closure are defined respectively by $I_{v}=(I^{-1})^{-1}$ and
$I_{t}=\displaystyle\bigcup J_{v}$ where $J$ ranges over the set of
f. g. subideals of $I$. We say that $I$ is divisorial (or a
$v$-ideal) if $I=I_{v}$, and a $t$-ideal if $I=I_{t}$. Unreferenced
material is standard as in \cite{G1} or \cite{Ka}.
\section{Ratliff-Rush ideals in an integral domain}\label{RRID}

Let $R$ be an integral domain. A nonzero ideal $I$ of $R$ is
$L$-stable (here $L$ stands for Lipman) if
$R^{I}:=\bigcup(I^{n}:I^{n})=(I:I)$. The ideal $I$ is stable (or
Sally-Vasconcelos stable) if  $I$ is invertible in its endomorphisms
ring $(I:I)$ (\cite{SV}). A domain $R$ is $L$-stable (respectively
stable) if every nonzero ideal of $R$ is $L$-stable (respectively
stable). We recall that a stable domain is $L$-stable \cite[Lemma
2.1]{AHP}, and for recent developments on stability (in settings
different than originally considered), we refer the reader to
\cite{AHP, O1, O2, O3}. We start this section with the following
definition which
extend the notion of Ratliff-Rush closure to nonzero integral ideals in an arbitrary integral domain.\\

\begin{definition}\label{RRID.1} Let $R$ be an integral domain and $I$ a nonzero
integral ideal of $R$. The Ratliff-Rush closure of $I$ is the
(integral) ideal of $R$ given by\\
$\tilde{I}=\displaystyle\bigcup(I^{n+1}:_{R}I^{n})$. An integral
ideal $I$ of $R$ is said to be a Ratliff-Rush ideal, or Ratliff-Rush
closed, if $I=\tilde{I}$, and $R$ is said to be a Ratliff-Rush
domain if each nonzero integral ideal of $R$ is a Ratliff-Rush
ideal.
\end{definition}

The following useful lemma treats the Ratliff-Rush closure of some
particular classes of ideals.\\

\begin{lemma}\label{RRID.2} Let $R$ be an integral domain. Then:\\
1-All stable (and thus all invertible) ideals are Ratliff-Rush.\\
2-If $I$ is a nonzero idempotent ideal of $R$, then $\tilde{I}=R$.\\
\end{lemma}

\begin{proof} 1) Let $I$ be a stable ideal of $R$ and set $T=(I:I)$.
Then $I(T:I)=T$. Now, let $x\in \tilde{I}$. Then $x\in R$ and
$xI^{s}\subseteq I^{s+1}$ for some positive integer $s$. Composing
the two sides with $(T:I)$ and using the fact that $I(T:I)=T$, we
obtain $xI^{s-1}\subseteq I^{s}$. Iterating this process, we get
$xT\subseteq I$. Hence $x\in I$ and therefore
$I=\tilde{I}$, as desired.\\
2) Let $I$ be a nonzero idempotent ideal of $R$. Then for each $n$,
$I^{n}=I$. So $(I^{n+1}:_{R}I^{n})=(I:_{R}I)=(I:I)\cap R=R$. Hence
$\tilde{I}=R$.\\
\end{proof}

\bigskip

The next proposition relates the Ratliff-Rush closure to the
$L$-stability.\\

\begin{proposition}\label{RRID.3}Let $R$ be an integral domain. If
$R$ is a Ratliff-Rush domain, then $R$ is $L$-stable.
\end{proposition}

\begin{proof} Assume that $R$ is a Ratliff-Rush domain. Let $I$ be a
nonzero (integral) ideal of $R$ and let $x\in R^{I}$. Then there
exists a positive integer $n$ such that $xI^{n}\subseteq I^{n}$. Let
$0\not =d\in R$ such that $dx\in R$.  Then $xI^{n+1}\subseteq
I^{n+1}$ implies that $dxI(dI)^{n}=d^{n+1}xI^{n+1}\subseteq
d^{n+1}I^{n+1}=(dI)^{n+1}$. Hence $dxI\subseteq
((dI)^{n+1}:(dI)^{n})$. Since $dxI\subseteq R$, then $dxI\subseteq
\widetilde{(dI)}=dI$ (since $R$ is Ratliff-Rush) and so $xI\subseteq
I$. Hence $x\in (I:I)$ and therefore $R^{I}=(I:I)$. So $I$ is
$L$-stable and therefore $R$ is $L$-stable, as desired.
\end{proof}

\bigskip

It's easy to see that for a finitely generated ideal $I$ of a domain
$R$, in particular if $R$ is Noetherian, $\tilde{I}\subseteq I'$.
However, this is not the case for an arbitrary ideal of an integral
domain. Indeed, let $V$ be a valuation domain with maximal ideal $M$
such that $M^{2}=M$, $0\not =a\in M$ and set $I=aM$. It is easy to
see that $\tilde{I}=a(M:M)\cap V=aV$ and $I=I'$ (since all ideals of
a Pr\"ufer domains are integrally closed). The next theorem establishes a
connection between stable domains, Ratliff-Rush domains and domains for which $\tilde{I}\subseteq I'$ for all ideals $I$.
For this, we need the following crucial lemma.\\

\begin{lemma}\label{RRID.4} Let $R$ be an integral domain.
If $\tilde{I}=I$ for every finitely generated ideal $I$ of $R$, then
$R'$ is a Pr\"ufer domain.
\end{lemma}

\begin{proof} Let $N$ be a maximal ideal of $R'$. To show that $R'_{N}$
is a valuation domain, let $x={a\over b}\in qf(R)$, where $a, b\in
R\setminus \{0\}$. Let $J$ be the ideal $(a^{4},a^{3}b, ab^{3},
b^{4})$ of $R$. Then $a^{2}b^{2}J=(a^{6}b^{2}, a^{5}b^{3},
a^{3}b^{5}, a^{2}b^{6})\subseteq J^{2}=(a^{8}, a^{7}b, a^{5}b^{3},
a^{4}b^{4}, a^{6}b^{2}, a^{3}b^{5},a^{2}b^{6}, ab^{7}, b^{8})$. So
$a^{2}b^{2}\in (J^{2}:_{R}J)\subseteq \tilde{J}=J$. Thus $
a^{2}b^{2}=g_{1}a^{4}+g_{2}a^{3}b+g_{3}ab^{3}+g_{4}b^{4}$ for some
$g_{1}, g_{2}, g_{3}$ and $g_{4}$ in $R$. Dividing by $b^{4}$, we
get $0=g_{1}x^{4}+g_{2}x^{3}-x^{2}+g_{3}x+g_{4}$. By the $u$,
$u^{-1}$ theorem (\cite[Theorem 67]{Ka}), we get that either $x\in
R'_{N}$ or $x^{-1}\in R'_{N}$, as desired.
\end{proof}

\bigskip

\begin{thm}\label{RRID.5} Let $R$ be an integral domain. Consider the following.\\
$(1)$ $R$ is stable.\\
$(2)$ $R$ is Ratliff-Rush.\\
$(3)$ $\tilde{I}\subseteq I'$ for each nonzero ideal $I$ of $R$.\\
$(4)$ $R$ has no nonzero idempotent prime ideals.\\
Then $(1)\Longrightarrow (2)\Longrightarrow (3)\Longrightarrow (4)$.
Moreover, If $R$ is a semilocal Pr\"ufer domain, then
$(4)\Longrightarrow (1)$.
\end{thm}

\begin{proof} $(1)\Longrightarrow (2)$ by Lemma~\ref{RRID.2}.\\
$(2)\Longrightarrow (3)$ is clear.\\
For $(3)\Longrightarrow (4)$, assume that $P$ is a nonzero
idempotent prime ideal of $R$. Then if $I=aP$ with $0\not =a\in P$,
then for all $n\geq 1$, $(I^{n+1}:_{R}I^{n})=(I^{n+1}:I^{n})\cap
R=(a^{n+1}P:a^{n}P)\cap R=a(P:P)\cap R=a(P:P)$ (since
$a(P:P)\subseteq P(P:P)=P\subseteq R$). So $a\in a(P:P)=\tilde{I}$.
Suppose $a\in I'=(aP)'$. Then $a^{k}+ c_{1}a^{k-1}+\dots + c_{k}=0$,
where $c_{i}=a^{i}b_{i}\in I^{i}=a^{i}P$ for each $i\in \{1, \dots,
k\}\}$. So $a^{k}+ b_{1}a^{k}+ b_{2}a^{k}+\dots + b_{k}a^{k}=0$ with
$b_{i}\in P$. Thus $a^{k+1}(1+b)=0$ with $b\in P$, a
contradiction.\\
$(4)\Longleftrightarrow (1)$ if $R$ is a semilocal Pr\"ufer domain
by \cite[Theorem 2.10]{AHP}.
\end{proof}

\bigskip

We are now ready to announce the main theorem of this section. It
gives a classification of the integral domains for which every ideal
is a Ratliff-Rush ideal in the context of integrally closed domains
 and states a new characterization of Pr\"ufer and strongly discrete
 Pr\"ufer domains. Recall that a Pr\"ufer domain is said to be strongly
 discrete if $P\not= P^{2}$ for each nonzero prime ideal $P$ of $R$.\\

\begin{thm}\label{RRID.6} Let $R$ be an integrally closed domain. The following
statements are equivalent.\\
$(1)$ $\tilde{I}=I$ for every finitely generated (respectively
every) nonzero ideal $I$ of $R$.\\
$(2)$ $R$ is Pr\"ufer (respectively strongly discrete Pr\"ufer).
\end{thm}

\begin{proof} $(1)\Longrightarrow (2)$ By
Lemma~\ref{RRID.4}, $R$ is a Pr\"ufer
domain. Moreover, if each ideal is a Ratliff-Rush ideal, by Theorem~\ref{RRID.5}, $R$ is strongly discrete.\\

$(2)\Longrightarrow (1)$. Let $R$ be a Pr\"ufer domain. Then every
finitely generated ideal is invertible and therefore a Ratliff-Rush
ideal by Lemma~\ref{RRID.2}. Assume that $R$ is a strongly discrete
Pr\"ufer domain. Let $I$ be a nonzero ideal of $R$ and let $x\in
\tilde{I}$. Then $x\in R$ and $xI^{s}\subseteq I^{s+1}$ for some
positive integer $s$. Let $M$ be a maximal ideal of $R$. If
$I\not\subseteq M$, then $x\in R\subseteq R_{M}=IR_{M}$. Assume that
$I\subseteq M$. Since $x\in R_{M}$ and $xI^{s}R_{M}\subseteq
I^{s+1}R_{M}$, then $x\in \widetilde{IR_{M}}$. Since $R$ is strongly
discrete, then $R_{M}$ is a strongly discrete valuation domain. By
Theorem~\ref{RRID.5}, $\widetilde{IR_{M}}=IR_{M}$. Hence $x\in
IR_{M}$. So $x\in \bigcap \{IR_{M}/M\in Max(R)\}=I$. Hence
$I=\tilde{I}$, as desired.\\
\end{proof}

The following example shows that the above Theorem is not true if
$R$ is not integrally closed.\\

\begin{example}\label{RRID.7} Let $\mathbb{Q}$ be the field of
rational numbers, $X$ an indeterminate over $\mathbb{Q}$ and
$V=\mathbb{Q}(\sqrt{2})[[X]]=\mathbb{Q}(\sqrt{2}) + M$. Set
$R=\mathbb{Q}+M$. Then $R$ is stable. Indeed, Let $I$ be a nonzero
(integral) ideal of $R$. Since $R$ is local with maximal ideal $M$,
then $I\subseteq M$. If $I$ is an ideal of $V$, then $I=cV$ for some
$c\in I$. If $I$ is not an ideal of $V$, then $I=m(W+M)$, where
$\mathbb{Q}\subseteq W\subsetneqq \mathbb{Q}(\sqrt{2})$ is a
$\mathbb{Q}$-vector space. Since
$[\mathbb{Q}(\sqrt{2}):\mathbb{Q}]=2$, then $\mathbb{Q}=W$ and so
$I=cR$. Therefore $R$ is stable and then Ratliff-Rush by
Theorem~\ref{RRID.5}. However, $R$ is not a Pr\"ufer domain
(\cite[Theorem 2.1]{BG}).
\end{example}

\bigskip

Our next corollary recovers Heinzer-Lantz-Shah's results for
Noetherian domains.\\

\begin{corollary}\label{RRID.8}(cf. \cite[Proposition 3.1 and
Theorem 3.9]{HLS} Let $R$ be a Noetherian domain. Then $R$ is a
Ratliff-Rush domain if and only if $R$ is stable.
\end{corollary}

\begin{proof} Since $R$ is Noetherian, then $R'=\bar{R}$ is a Krull
domain. By Lemma~\ref{RRID.4}, $R'$ is a Pr\"ufer domain. Hence $R'$
is a Dedekind domain and therefore $dimR=dimR'=1$. By
Proposition~\ref{RRID.3}, $R$ is $L$-stable and therefore stable by
\cite[Proposition 2.4]{AHP}.
\end{proof}

We recall that a domain $R$ is said to be strong Mori if $R$
satisfies the ascending chain conditions on $w$-ideals \cite {FM}.
Trivially, a Noetherian domain is strong Mori and a strong Mori
domain is Mori. The next corollary shows that the
Ratliff-Rush property forces a strong Mori domain to be Noetherian.\\

\begin{corollary}\label{RRIP.5} Let $R$ be a strong Mori domain. If
$R$ is a Ratliff-Rush domain, then $R$ is Noetherian.
\end{corollary}

\begin{proof} By Lemma~\ref{RRID.4}, $R'$ is a Pr\"ufer domain. Hence every maximal ideal of $R$ is
divisorial (\cite[Corollary 2.5]{DHLRZ} and \cite[Theorem
2.6]{DHLZ}). Now, let $M$ be a maximal ideal of $R$. Since
$M=M_{v}$, then $R_{M}$ is Noetherian (\cite[Theorem 3.9]{FM}).
Hence $R'_{M}=(R_{M})'=\overline{R_{M}}$ is a Krull domain. But
since $R'$ is Pr\"ufer, then so is $R_{M}'$. Hence $R_{M}'$ is
Dedekind and so $htM=dimR_{M}=dimR_{M}'=1$. Then $dimR=1$ and
therefore $R$ is Noetherian (\cite[Corollary 3.10]{FM}).
\end{proof}

Recall that $R$ is seminormal if for each $x\in qf(R)$, $x^{2},
x^{3}\in R$ implies that $x\in R$. Our next corollary states
some conditions under which a Ratliff-Rush Mori domain has dimension one.\\

\begin{corollary}\label{RRID.9} Let $R$ be a Mori domain such that either
$(R:\overline{R})\not =0$ or $R$ is seminormal. If $R$ is a
Ratliff-Rush domain, then $dimR=1$.
\end{corollary}

\begin{proof} Assume that $R$ is a Ratliff-Rush domain. By
Lemma~\ref{RRID.4}, $R'$ is a Pr\"ufer domain.\\
\1 If $(R:\overline{R})\not=(0)$, then $\overline{R}$ is a Krull
domain (\cite[Corollary 18]{Bar}). Since $R'\subseteq \overline{R}$,
then $\overline{R}$ is a Pr\"ufer domain, and therefore Dedekind.
Hence $\dim(\overline{R})=1$. By \cite[Corollary 3.4]{BH},
$\dim(R)=1$, as desired.\\
\2 Assume that $R$ is seminormal. If $\dim(R)\geq 2$, then $R$ has a
maximal ideal $M$ such that $htM\geq 2$. Set
$B=(MR_{M})^{-1}=(MR_{M}:MR_{M})$. Since $R_{M}$ is a local Mori
domain which is seminormal and $htMR_{M}=htM\geq 2$, then $B$
contains a nondivisorial maximal ideal $N$ contracting to $MR_{M}$
(\cite[Lemma 2.5]{BH}). Since $R'$ is a Pr\"ufer domain
(Lemma~\ref{RRID.4}) and combining \cite[Corollary 2.5]{DHLRZ}) and
\cite[Theorem 2.6]{DHLZ}, we get that every maximal ideal of $B$ is
a $t$-ideal and so a $v$-ideal since $B$ is Mori, which is absurd.
Hence $\dim(R)=1$, as desired.
\end{proof}

\section{Ratliff-Rush ideals in a Valuation domain}\label{RRIV}

It's well-known that the maximal ideal $M$ of a valuation domain $V$
is either principal or idempotent, any nonzero prime ideal $P$ of
$V$ is a divided prime ideal, that is, $PV_{P}=P$, and any
idempotent ideal is a prime ideal. Also we recall that a valuation
domain is a $TP$ domain, that is, for each nonzero ideal $I$ of $V$,
either $II^{-1}=V$ or $II^{-1}=Q$ is a prime ideal of $V$
(\cite[Proposition 2.1]{FHP}), and for each positive integer $n$,
$I^{n}I^{-n}=II^{-1}$ (\cite[Remark 2.13(b)]{HP}). We will often use
this facts without explicit mention. Finally $V$ is strongly
discrete if it has no nonzero idempotent prime ideal (\cite[chapter
5.3]{FHP1}).\\

\bigskip

\begin{lemma}\label{RRIV.1}
Let $V$ be a valuation domain, $I$ a nonzero ideal of $V$ and assume
that $\tilde{I}\not =V$. Then $(I:I)\subseteq
(\tilde{I}:\tilde{I})$.
\end{lemma}

\begin{proof} Let $I$ be a nonzero ideal of $V$ and assume that
$\tilde{I}\not =V$. If $II^{-1}=V$, then $I=\tilde{I}$ by
Lemma~\ref{RRID.2} and therefore $(I:I)= (\tilde{I}:\tilde{I})$.
Assume that $II^{-1}=Q$ is a prime ideal of $V$. Since $V$ is a
valuation domain, then $V$ is $L$-stable. So $(I:I)=(I^{n}:I^{n})$
for each positive integer $n$. Let $x\in (I:I)$ and $z\in
\tilde{I}$. Then $z\in V$ and $zI^{r}\subseteq I^{r+1}$ for some
positive integer $r$. Since $(I:I)=(I^{r+1}:I^{r+1})$, then
$xzI^{r}\subseteq xI^{r+1}\subseteq I^{r+1}$. Hence $xz\in
(I^{r+1}:I^{r})$. To show that $xz\in \tilde{I}$, it suffices to
prove that $xz\in V$. Suppose that $xz\not\in V$. Then $(xz)^{-1}\in
V$. Since $z\in \tilde{I}$, then $x^{-1}=(xz)^{-1}z\in \tilde{I}$.
So $x^{-1}\in V$ and $x^{-1}I^{s}\subseteq I^{s+1}$ for some
positive integer $s$. Hence $I^{s}\subseteq xI^{s+1}\subseteq
I^{s+1}$ (since $(I:I)=(I^{s+1}:I^{s+1})$) and therefore
$I^{s}=I^{s+1}$. Hence $I^{s}=I^{2s}$ and therefore $I=P$ is an
idempotent prime ideal of $V$. By Lemma~\ref{RRID.2},
$\tilde{I}=\tilde{P}=V$, which is absurd. Hence $xz\in V$. So $xz\in
\tilde{I}$ and then $x\tilde{I}\subseteq \tilde{I}$. Hence $x\in
(\tilde{I}:\tilde{I})$ and therefore $V_{Q}=(I:I)\subseteq (\tilde{I}:\tilde{I})$.\\
\end{proof}

The next proposition describes the Ratliff-Rush closure of a nonzero
integral ideal in a valuation domain.\\

\begin{proposition}\label{RRIV.2} Let $I$ be a nonzero integral
ideal of a valuation domain $V$. Then:\\
\1 $\tilde{I}=V$ if and only if $I$ is an idempotent prime ideal of $V$.\\
\2 Assume that $\tilde{I}\subsetneq V$. Then either $\tilde{I}=I$,
or $\tilde{I}=(IQ:_{V}Q)$ for some nonzero prime ideal $Q$ of $V$.
\end{proposition}

\begin{proof} \1 If $I$ is an idempotent prime ideal of $V$, by
 Lemma~\ref{RRID.2}, $\tilde{I}=V$. Conversely, assume that
 $\tilde{I}=V$. Then there exists a positive integer $n$ such that
 $I^{n}\subseteq I^{n+1}$. Hence $I^{n}=I^{n+1}$. By induction,
 $(I^{n})^{2}=I^{n}$. So $I^{n}$ is an idempotent ideal of $V$. Hence
 $I^{n}=P$ is a prime ideal of $V$. Then $I\subseteq P\subseteq I$ and therefore $I=P$, as
 desired.\\

\2 Assume that $\tilde{I}\subsetneq V$. If $II^{-1}=V$, then
$I=\tilde{I}$ by Lemma~\ref{RRID.2}. Assume that
$II^{-1}=Q\subsetneq V$ is a prime ideal. Then $(I:I)=V_{Q}$ and for
each positive integer $n$, $I^{n}I^{-n}=Q$ since $V$ is a $TP$
domain. Let $x\in \tilde{I}$. Then $x\in V$ and $xI^{n}\subseteq
I^{n+1}$ for some positive integer $n$. So $xQ=xI^{n}I^{-n}\subseteq
xI^{n+1}I^{-n}=IQ$. Hence $x\in (IQ:_{V}Q)$ and therefore
$\tilde{I}\subseteq (IQ:_{V}Q)$. Now, assume that $I\subsetneq \tilde{I}\subsetneq V$.\\
To complete the proof, we will show that $\tilde{I}=(IQ:_{V}Q)$.
Since $V_{Q}=(I:I)\subseteq (\tilde{I}:\tilde{I})$
(Lemma~\ref{RRIV.1}), then $\tilde{I}$ is an ideal of $V_{Q}$.
Suppose that $\tilde{I}\subsetneq (IQ:_{V}Q)$. Let $x\in
(IQ:_{V}Q)\setminus \tilde{I}$. Since $V$ is a valuation domain,
then $\tilde{I}\subsetneq xV$. So $x^{-1}\tilde{I}\subsetneq
V\subseteq V_{Q}$. Hence $x^{-1}\tilde{I}$ is a proper ideal of
$V_{Q}$. So $x^{-1}\tilde{I}\subseteq Q$ ($Q=QV_{Q}$ is the maximal
ideal of $V_{Q}$). Hence $\tilde{I}\subseteq xQ\subseteq IQ\subseteq
I\subsetneq \tilde{I}$, a contradiction. It follows that
$\tilde{I}=(IQ:_{V}Q)$, as desired.
\end{proof}

Our next proposition shows that the Ratliff-Rush closure of an ideal
$I$ in a valuation domain is itself a Ratliff-Rush ideal, and gives
necessary and sufficient condition for preserving the Ratliff-Rush
closure under inclusion.\\

\begin{proposition}\label{RRIV.3} Let $I$ be a nonzero ideal of a
valuation domain $V$. Then\\
1) $\tilde{\tilde{I}}=\tilde{I}$.\\
2) $\tilde{I}\subseteq \tilde{J}$ for every ideals $I\subseteq J$ if
and only each nonzero nonmaximal prime ideal of $V$ in not
idempotent.
\end{proposition}

\begin{proof} 1) If $I=\tilde{I}$ or $\tilde{I}=V$, then clearly
$\tilde{\tilde{I}}=\tilde{I}$. Assume that $I\subsetneq
\tilde{I}\subsetneq V$. By Proposition~\ref{RRIV.2},
$\tilde{I}=(IQ:_{V}Q)$ where $Q=II^{-1}$ is a prime ideal of $V$
(note that $II^{-1}\subsetneq V$, otherwise $I=\tilde{I}$, by
Lemma~\ref{RRID.2}). For simplicity, we set $J=\tilde{I}$. Our aim
is to prove that $J=\tilde{J}$. If $JJ^{-1}=V$, then $J=\tilde{J}$
by Lemma~\ref{RRID.2}. Assume that $JJ^{-1}\subsetneq V$. By
Lemma~\ref{RRIV.1}, $V_{Q}=(I:I)\subseteq
(\tilde{I}:\tilde{I})=(J:J)= V_{P}$, where $P=JJ^{-1}$. So
$P\subseteq Q$. Let $x\in \tilde{J}$. Then $x\in V$ and
$xJ^{n}\subseteq J^{n+1}$ for some positive integer $n$. Composing
the two sides with $J^{-n}$ and using the fact that
$P=JJ^{-1}=J^{n}J^{-n}$, we obtain $xP\subseteq JP$. Hence
$\tilde{J}P\subseteq JP\subseteq JQ=\tilde{I}Q=IQ$. Now, if
$P\subsetneq Q$, then let $a\in Q\setminus P$. Since $V$ is a
valuation domain, then $P\subsetneq aV$. So $a^{-1}P\subsetneq V$.
Hence $a^{-1}\in (V:P)=(P:P)=V_{P}=(J:J)$ (\cite{HuP}). So
$a^{-1}J\subseteq J$. Then $J\subseteq aJ\subseteq QJ=QI\subseteq
I\subsetneq J$, a contradiction. Hence $P=Q$. So
$\tilde{J}P=\tilde{J}Q=JQ=IQ$. Hence
$\tilde{J}\subseteq (IQ:_{V}:Q)=\tilde{I}=J$, as desired.\\

2) Assume that $\tilde{I}\subseteq \tilde{J}$ for every ideals
$I\subseteq J$. Suppose that there is a nonzero nonmaximal prime
ideal $P$ of $V$ such that $P^{2}=P$. Let $a\in M\setminus P$, where
$M$ is the maximal ideal of $V$. Since $V$ is a valuation domain,
then $P\subsetneq aV=I$. By Lemma~\ref{RRID.2} and the hypothesis,
$V=\tilde{P}\subseteq \tilde{I}=aV\subseteq M$, which is absurd.\\

Conversely, assume that each nonzero nonmaximal prime ideal of $V$
in not idempotent and let $I\subseteq J$ be ideals of $V$. If
$I=\tilde{I}$, or $\tilde{J}=V$, then clearly $\tilde{I}\subseteq
\tilde{J}$. If $\tilde{I}=V$, by Proposition~\ref{RRIV.2}, $I=P$ is
an idempotent prime ideal of $V$. By the hypothesis, $I=M$. So
$M=I\subseteq J\subseteq M$. Then $I=J=M$ and so
$\tilde{I}=\tilde{J}$. Hence we may assume that $I\subsetneq
\tilde{I}\subsetneq V$ and $\tilde{J}\subsetneq V$. By
Proposition~\ref{RRIV.2}, $\tilde{I}=(IQ:_{V}Q)$, where $Q=II^{-1}$.
Now, suppose that $\tilde{I}\not\subseteq \tilde{J}$. Then let $x\in
\tilde{I}\setminus \tilde{J}$. Since $V$ is a valuation domain, then
$\tilde{J}\subsetneq xV$. So $x^{-1}I\subseteq x^{-1}J\subseteq
x^{-1}\tilde{J}\subsetneq V\subseteq V_{Q}$. Since $I$ is an ideal
of $(I:I)=V_{Q}$, then $x^{-1}I\subseteq Q$. So $I\subseteq
xQ\subseteq \tilde{I}Q=IQ\subseteq I$. Therefore $I=xQ$. If $Q$ is
nonmaximal, by the hypothesis, $Q^{2}\subsetneq Q$. Hence $Q=aV_{Q}$
for some nonzero $a\in Q$ (since $Q$ is the maximal ideal of
$V_{Q}$). Hence $I=xQ=xaV_{Q}=xa(I:I)$. So $I$ is stable and by
Lemma~\ref{RRID.2}, $\tilde{I}=I$, which is absurd. Hence $Q=M$ and
$I=xM$. If $M$ is principal in $V$, then so is $I$ and therefore
$\tilde{I}=I$, which is absurd. Hence $M=M^{2}$. So
$\tilde{I}=(IM:_{V}M)=(xM^{2}:_{V}M)=(xM:_{V}M)=x(M:M)=xV$. Let
$b\in J\setminus I$. Then $xM=I\subsetneq bV$. Hence
$xb^{-1}M\subseteq M$. So $xb^{-1}\in (M:M)=V$. Hence
$x=(xb^{-1})b\in J\subseteq \tilde{J}$, which is absurd. It follows
that $\tilde{I}\subseteq \tilde{J}$, as desired.
\end{proof}

\bigskip

Now, we extend the Ratliff-Rush closure to arbitrary nonzero
fractional ideals and we study its link to the notion of star
operations. Our motivation is \cite[Example 1.11]{HLS}, which
provided an example of a Noetherian domain $R$ with a nonzero ideal
$I$ such that $\widetilde{aI}\not=a\tilde{I}$ for some $0\not=a\in
R$. First, we recall that a star operation on $R$ is a map $*:
F(R)\longrightarrow F(R), E\mapsto E^{*}$, where $F(R)$ denotes the
set of all nonzero fractional ideals of $R$, with the following
properties for each $E, F\in F(R)$ and each $0\not
=a\in K$:\\
$(E_{1})$ $R^{*}=R$ and $(aE)^{*}=aE^{*}$;\\
$(E_{2})$ $E\subseteq E^{*}$ and if $E\subseteq F$, then
$E^{*}\subseteq F^{*}$;\\
$(E_{3})$ $E^{**}=E^{*}$.\\
For more details on the notion of star operations, we refer the reader to \cite{G1}. \\

\begin{definition}\label{RRIV.4} Let $R$ be an integral domain with
quotient field $K$ and let $I$ be a nonzero fractional ideal of
$R$.\\
$(1)$ The generalized Ratliff-Rush closure of $I$ is defined by
$\hat{I}:=\{x\in K| xI^{n}\subseteq I^{n+1}$, for some $n\geq 1\}$.
Clearly $\tilde{I}=\hat{I}\cap R$ for any nonzero integral ideal $I$
of $R$.
\end{definition}

It is easy to see that for a nonzero fractional ideal $I$ of a
domain $R$, $\hat{I}$ is an $R$-module which is a fractional ideal
if $(R:R^{I})\not=0$. In particular if $R$ is conducive or
$L$-stable, then $\hat{I}$ is always a fractional ideal of $R$. The
next theorem gives necessary and sufficient conditions for the
generalized Ratliff-Rush closure to be a star operation on a
valuation domain.

\begin{thm}\label{RRIV.5} Let $V$ be a valuation domain. The
generalized Ratllif-Rush closure on $V$ is a star operation if and
only if each nonzero nonmaximal prime ideal $P$ of $V$ is not
idempotent. In this case, it coincides with the $v$-operation.
\end{thm}

\begin{proof} Assume that the generalized Ratliff-Rush closure is a
star operation. Then, by Proposition~\ref{RRIV.3}, each nonzero
nonmaximal prime ideal of $V$ is not idempotent. Conversely, assume
that each nonzero nonmaximal prime ideal of $V$ is not idempotent.\\
{\bf Claim}. For each integral ideal $I$ of $V$,
$\tilde{I}=\hat{I}$. Indeed, it suffices to show that
$\hat{I}\subseteq V$.  If $II^{-1}=V$, then $\hat{I}=I$, as desired.
Assume that $II^{-1}=Q$ is a prime ideal of $V$. Then $(I:I)=V_{Q}$.
Let $x\in \hat{I}$. Then $xI^{n}\subseteq I^{n+1}$ for some positive
integer $n$. Since $I^{n}I^{-n}=Q$, we get $xQ\subseteq IQ$. Now, if
$Q=M$, then $xM\subseteq IM\subseteq M$. So $x\in (M:M)=V$. If
$Q\subsetneq M$, by hypothesis, $Q$ is not idempotent. Hence
$Q=aV_{Q}$ (since $Q$ is the maximal ideal of $V_{Q}$). So
$xaV_{Q}\subseteq aIV_{Q}=aI$ (here $I$ is an ideal of
$(I:I)=V_{Q}$). Hence $xV_{Q}\subseteq I$ and therefore $x\in
I\subseteq V$, as desired.\\
Now, we prove the three properties of star operations.
Let $I$ and $J$ be nonzero fractional ideals of $V$ and $o\not=a\in qf(V)$.\\
\1 $(E_{1})$: $x\in \widehat{aI}$ if and only if $x(aI)^{n}\subseteq
(aI)^{n+1}$ for some positive integer $n$, if and only if
$xa^{-1}\in (I^{n+1}:I^{n})\subseteq \hat{I}$, if and only if $x\in
a\hat{I}$.\\
\2 $(E_{2})$: Let $o\not =d\in V$ such that $dI\subseteq dJ\subseteq
V$. By $(E_{1})$, Proposition~\ref{RRIV.3}(2) and the claim,
$d\hat{I}=\widehat{dI}=\widetilde{dI}\subseteq
\widetilde{dJ}=\widehat{dJ}=d\hat{J}$. Hence $\hat{I}\subseteq \hat{J}$.\\
\3 $(E_{3})$: Clearly $I\subseteq \hat{I}$ and by $(E_{1})$ and
Proposition~\ref{RRIV.3}(1), $\hat{\hat{I}}=\hat{I}$.\\
To complete the proof, we prove that $\tilde{I}=I_{v}$ for each
nonzero fractional ideal $I$ of $V$. Since the $v$-operation is the
largest star operation on $V$, then $\hat{I}\subseteq I_{v}$.
Suppose that $\hat{I}\subsetneq I_{v}$ for some ideal $I$ of $V$.
Then $I$ is not divisorial in $V$. Hence $I=aM$ for some $a\in
qf(V)$ and $M=M^{2}$. Since $M$ is idempotent, then $M$ is not
divisorial. So $M_{v}=V$. Hence $I_{v}=aM_{v}=aV=\hat{I}$ (note that
by $(E_{1})$ and Lemma~\ref{RRID.2}
$\hat{I}=a\hat{M}=a\tilde{M}=aV$), which is absurd.
\end{proof}




\end{document}